\documentclass[a4paper,12pt]{article}
\usepackage[active]{srcltx}
\usepackage{amsfonts}
\usepackage{amsmath}
\usepackage{amsbsy}
\usepackage{amsxtra}
\usepackage{latexsym}
\usepackage{amssymb}
\usepackage{url}
\usepackage{pstricks,pst-node}
\usepackage{cite}
\usepackage{enumerate}
\makeatletter
\g@addto@macro\thesection.
\makeatother

\newtheorem{theorem}{Theorem}
\newtheorem{lemma}{Lemma}
\newtheorem{corollary}{Corollary}

\newtheorem{remark}{Remark}

\newtheorem{example}{Example}

\DeclareMathOperator{\cone}{cone}

\DeclareMathOperator{\spa}{span}

\DeclareMathOperator{\dme}{dim}

\newcommand{\lng}{\left\langle}
\newcommand{\rng}{\right\rangle}
\newcommand{\lf}{\left}
\newcommand{\rg}{\right}

\newcommand{\sa}{\sqcap}
\newcommand{\su}{\sqcup}
\newcommand{\R}{\mathbb R}
\newcommand{\smp}{\mathbb S^m_+}
\newcommand{\sm}{\mathbb S^m}

\newcommand{\n}{\mathbb N}
\newcommand{\mc}{\mathcal}

\newcommand{\ds}{\displaystyle}

\newcommand{\m}{\mathbf}

\newenvironment{proof}{{\noindent\bf Proof.}}{\hfill$\Box$\\}

\begin{document}

\title{Lattice-like subsets of Euclidean Jordan algebras
\thanks{{\it 1991 A M S Subject Classification:} Primary 90C33, Secondary 15A48; {\it Key words and phrases:} positive semidefinite cone, extended lattice operations, isotone projection onto a closed convex set, invariant sets with respect to extended lattice operations, variational inequalities.}}
\author{A. B. N\'emeth\\Faculty of Mathematics and Computer Science\\Babe\c s Bolyai University, Str. Kog\u alniceanu nr. 1-3\\RO-400084 Cluj-Napoca, Romania\\email: nemab@math.ubbcluj.ro 
\and%
S. Z. N\'emeth\\School of Mathematics, The University of Birmingham\\The Watson Building, Edgbaston\\Birmingham B15 2TT, United Kingdom\\email: nemeths@for.mat.bham.ac.uk}
\date{}
\maketitle

\begin{abstract}
While studying some properties of linear operators in a Euclidean Jordan algebra,
Gowda, Sznajder and Tao have introduced generalized lattice operations based on the projection 
onto the cone of squares.
In two recent papers of the authors of the present paper it has been shown that these
lattice-like operators and their generalizations are important tools in establishing the isotonicity
of the metric projection onto some closed convex sets. The results of this kind are
motivated by metods for proving the existence of solutions of variational inequalities and methods 
for finding these solutions in a recursive way.
It turns out, that
the closed convex sets admitting isotone projections are exactly the sets
which are invariant with respect to these lattice-like operations, 
called lattice-like sets. 
In this paper it is shown that the Jordan subalgebras are lattice-like 
sets, but the converse in general is not true. In the case of simple Euclidean
Jordan algebras of rank at least three the lattice-like property
is rather restrictive, e.g., there are no lattice-like
proper closed convex sets with interior points.
\end{abstract}
 
\section{Introduction}

By using and generalizing the extended lattice operations due to Gowda, Sznajder and Tao 
\cite{GowdaSznajderTao2004}, in \cite{NemethNemeth2012a} and \cite{NemethNemeth2012b}
it has been shown that the projection onto a closed convex set is
isotone with respect to the order defined by a cone if and only if the set is invariant with 
respect to the extended lattice operations defined by the cone. We shall call such a set simply
invariant with respect to the cone, or if there is no ambiguity, lattice-like, or shortly l-l. We also showed that 
the a closed convex set with interior points is l-l if and only if all of its tangent hyperplanes are l-l. 
These results were motivated by iterative methods for variational
inequalities similar to the ones for complementarity problems in 
\cite{IsacNemeth1990b,IsacNemeth1990c,Nemeth2009a,AbbasNemeth2011}. More specifically, a 
variational inequality defined by a closed convex set $C$ and
a function $f$ can be equivalently written as the fixed point problem $\m x=P_C(\m x-f(\m x))$, where 
$P_C$ is the projection onto the closed convex set $C$. If the Picard iteration 
$\m x_{k+1}=P_C(\m x_k-f(\m x_k))$ is convergent and $f$ continuous, then the limit of $\m x_k$ is a 
solution of the variational inequality defined by $f$ and $C$. Therefore, it is important to 
give conditions under which the Picard iteration is convergent. This idea has been exploited 
in several papers, such as \cite{Auslender1976,Bertsekas1989,Iusem1997,Khobotov1987,Korpelevich1976,Marcotte1991,Nagurney1993,Sibony1970,Solodov1999,Solodov1996,Sun1996}. However, none of these
papers used the monotonicity of the sequence $\m x_k$. If one can show that 
$\m x_k$ is 
monotone increasing (decreasing) and bounded from above (below) with respect to an order 
defined by a regular cone (that is, a cone for which all such sequences are convergent), then 
it is convergent and its limit is a solution of the variational inequality defined by $f$ and 
$C$. In \cite{IsacNemeth1990b,IsacNemeth1990c,Nemeth2009a,AbbasNemeth2011} the convergence of the 
sequence $\m x_k$ was proved by using its monotonicity. Although they use 
non-iterative methods, we also mention the paper of H. Nishimura and E. A. Ok 
\cite{NishimuraOk2012}, where the isotonicity of the projection onto a closed convex set is used
for studying the solvability of variational inequalities and related equilibrium problems. 
To further accentuate the importance of ordered vector structures let us also mention that 
recently they are getting more and more ground in studying various fixed point and related 
equilibrium problems (see the book \cite{CarlHeikilla2011} of S. Carl and S Heikkil\"a and the 
references therein). The case of a self-dual 
cone is of special importance because of the elegant examples for invariant sets with respect 
to the nonnegative orthant and the Lorentz cone \cite{NemethNemeth2012a}. 
Moreover, properties of self-dual cones are becoming increasingly important because of conic 
optimization and applications of the analysis on symmetric cones. Especially important 
self-dual cones in applications are the nonnegative orthant, the Lorentz cone and the positive 
semidefinite cone, however the class of self-dual cones is much larger \cite{BarkerForan1976}. 
The results of \cite{NemethNemeth2012a} and \cite{NemethNemeth2012b} extend the results
of \cite{Isac1996} and \cite{NishimuraOk2012}. G. Isac showed in \cite{Isac1996} that the 
projection onto a closed convex sublattice of the Euclidean space ordered by the nonnegative 
orthant is isotone. H. Nishimura and E. A. Ok proved an extension of this  
result and its converse to Hilbert spaces in \cite{NishimuraOk2012}. The study of invariant sets with respect to
the nonnegative orthant goes back to the results of D. M. Topkis \cite{Topkis1976} and A. F.
Veinott Jr. \cite{Veinott1981}, but it wasn't until quite recently when all such invariant 
sets have been determined by M. Queyranne and F. Tardella \cite{QueyranneTardella2006}. The
same results have been obtained in \cite{NemethNemeth2012a} in a more geometric way. Although 
\cite{NemethNemeth2012a} also determined the invariant sets with respect to the Lorentz cone, 
it left open the question of finding the invariant sets with respect to the cone $\smp$ of 
$n\times n$ positive semidefinite matrices, called the positive semidefinite cone.

As a particular case we show that if 
$n\geq 3$, then there is no proper closed convex l-l set with nonempty interior  in 
the space $(\sm,\smp)$ (the space $\sm$ of $n\times n$ symmetric matrices ordered by the
cone $\smp$ of symmetric positive semidefinite matrices). For this it is 
enough to show that there are no invariant hyperplanes because the closed convex invariant sets 
with nonempty interior are the ones which have all tangent hyperplanes invariant. 

All these problems can be handled in the unifying context of the Euclidean Jordan algebras.
This way we can augment this field to an approach, where the order induced
by the cone of squares (the basic notion of the Jordan algebra) becomes emphasized.

To shorten our exposition, we assume the knowledge of basic facts and results on
Euclidean Jordan algebras. We strive to be in accordance with the terminology
in \cite{FarautKoranyi1994}. A concise introduction of the used
basic notions and facts in the field can be found in \cite{GowdaSznajderTao2004}.

\section{Preliminaries}

Denote by $\R^m$ the $m$-dimensional Euclidean space endowed with the scalar 
product $\lng\cdot,\cdot\rng:\R^m\times\R^m\to\R,$ and the Euclidean norm $\|\cdot\|$ and topology 
this scalar product defines.

Throughout this note we shall use some standard terms and results from convex geometry 
(see e.g. \cite{Rockafellar1970} and \cite{Zarantonello1971}). 

Let $K$ be a \emph{convex cone} in $\R^m$, i.e., a nonempty set with
(i) $K+K\subset K$ and (ii) $tK\subset K,\;\forall \;t\in \R_+ =[0,+\infty)$.
The convex cone $K$ is called \emph{pointed}, if $K\cap(-K)=\{\m 0\}.$

The convex cone $K$ is {\it generating} if $K-K=\R^m$.
 
For any $\m x,\m y\in \R^m$, by the equivalence $\m x\leq_K\m y\Leftrightarrow\m  y-\m x\in K$, the 
convex cone $K$ 
induces an {\it order relation} $\leq_K$ in $\R^m$, that is, a binary relation, which is 
reflexive and transitive. This order relation is {\it translation invariant} 
in the sense that $\m x\leq_K\m  y$ implies $\m x+\m z\leq_K\m  y+\m z$ for all $\m z\in \R^m$, and 
{\it scale invariant} in the sense that $\m x\leq_K\m y$ implies $t\m x\leq_K t\m y$ for any $t\in \R_+$.
If $\leq$ is a translation invariant and scale invariant order relation on $\R^m$, then 
$\leq=\leq_K$, where $K=\{\m x\in\R^m:\m 0\leq\m  x\}$ is a convex cone. If $K$ is pointed, then $\leq_K$ is 
\emph{antisymmetric} too, that is $\m x\leq_K\m  y$ and $\m y\leq_K\m  x$ imply that $\m x=\m y.$
The elements $\m x$ and $\m y$ are called \emph{comparable} if $\m x\leq_K\m  y$
or $\m y\leq_K\m  x.$

We say that $\leq_K$ is a \emph{latticial order} if for each pair
of elements $\m x,\m y\in \R^m$ there exist the least upper bound
$\sup\{\m x,\m y\}$ and the greatest lower bound $\inf\{\m x,\m y\}$ of
the set $\{\m x,\m y\}$ with respect to the order relation $\leq_K$.
In this case $K$ is said a \emph{latticial or simplicial cone},
and $\R^m$ equipped with a latticial order is called an
\emph{Euclidean vector lattice}.

The \emph{dual} of the convex cone $K$ is the set
$$K^*:=\{\m y\in \R^m:\;\lng\m  x,\m y\rng \geq 0,\;\forall \;\m x\in K\},$$
with $\lng\cdot,\cdot\rng $ the standard scalar product in $\R^m$.

The convex cone $K$ is called \emph{self-dual}, if $K=K^*.$ If $K$
is self-dual, then it is a generating pointed closed convex cone.

In all that follows we shall suppose that $\R^m$ is endowed with a
Cartesian reference system with a basis $\m e_1,\dots,\m e_m$. If $\m x\in \R^m$, then
$$\m x=x_1 \m e_1+...+x_m \m e_m$$ can be characterized by the ordered $m$-tuple of real numbers 
$x_1,...,x_m$, called
\emph{the coordinates of} $x$ with respect to the given reference system, and we shall write
$\m x=(x_1,...,x_m).$ With this notation we have $\m e_i=(0,...,0,1,0,...,0),$
with $1$ in the $i$-th position and $0$ elsewhere. Let
$\m x,\m y\in \R^m$, $\m x=(x_1,...,x_m)$, $\m y=(y_1,...,y_m)$, where $x_i$, $y_i$ are the coordinates of
$\m x$ and $\m y$, respectively with respect to the reference system. Then, the scalar product of $\m x$
and $\m y$ is the sum
$\lng\m  x,\m y\rng =\sum_{i=1}^m x_iy_i.$
It is easy to see that $\m e_1,\dots,\m e_m$ is an orthonormal system of vectors with respect to 
this scalar product, in the sense
that $\lng \m e_i,\m e_j\rng =\delta_i^j$, where $\delta_i^j$ is the Kronecker symbol.

The set
\[\R^m_+=\{\m x=(x_1,...,x_m)\in \R^m:\; x_i\geq 0,\;i=1,...,m\}\]
is called the \emph{nonnegative orthant} of the above introduced Cartesian
reference system. A direct verification shows that $\R^m_+$ is a
self-dual cone.

The set
\begin{equation}\label{lorentzcone}
	\mathcal L^{m+1}_+=\{(\m x,x_{m+1})\in\R^{m}\otimes\R=\R^{m+1}:\;\|\m x\|\leq x_{m+1}\},
\end{equation} 
is a self-dual cone called the 
\emph{$m+1$-dimensional second order cone}, or the 
\emph{$m+1$-dimensional Lorentz cone}, or the \emph{$m+1$-dimensional ice-cream cone} 
\cite{GowdaSznajderTao2004}. 

The nonnegative orthant $\R^m_+$ and the Lorentz cone $L$ defined above are the
most important and commonly used self-dual cones in the Euclidean space.
But the family of self-dual cones is rather rich
\cite{BarkerForan1976}.

\section{Generalized lattice operations}\label{Generalized lattice operations}

A  \emph{hyperplane through the origin}, is a set of form
\begin{equation}\label{hypersubspace}
H(\m 0,\m a)=\{\m x\in \R^m:\;\lng\m  a,\m x\rng =0\},\;\;\m a\not=\m 0.
\end{equation}
For simplicity the hyperplanes through $\m 0$ will also be denoted by $H$.
The nonzero vector $\m a$ in the above formula is called \emph{the normal}
of the hyperplane. 

A \emph{hyperplane through $\m u\in\R^m$ with the normal $\m a$} is the set of the form
\begin{equation}\label{hyperplane}
H(\m u,\m a)=\{\m x\in \R^m:\;\lng\m  a,\m x\rng =\lng\m  a,\m u\rng, \;\m a\not=\m  0\}.
\end{equation}

A hyperplane $H(\m u,\m a)$ determines two \emph{closed halfspaces} $H_-(\m u,\m a)$ and
$H_+(\m u,\m a)$  of $\R^m$, defined by

\[H_-(\m u,\m a)=\{\m x\in \R^m:\; \lng\m  a,\m x\rng \leq \lng\m  a,\m u\rng\},\]
and
\[H_+(\m u,\m a)=\{\m x\in \R^m:\; \lng\m  a,\m x\rng \geq \lng\m  a,\m u\rng\}.\]

Taking a Cartesian reference system in $\R^m$ and using the above introduced
notations,
the 
\emph{coordinate-wise order}  $\leq$ in $\R^m$ is defined by
\[\m x=(x_1,...,x_m)\leq\m  y=(y_1,...,y_m)\;\Leftrightarrow\;x_i\leq y_i,\;i=1,...,m.\]
By using the notion of the order relation induced by a cone, defined in the preceding
section, we see that $\leq =\leq_{\R^m_+}$.

With the above representation of $\m x$ and $\m y$, we define
$$\m x\wedge\m  y=(\min \{x_1,y_1\},...,\min \{x_m,y_m\}),\;\;\textrm{and}\;\;\m x\vee\m  y=(\max \{x_1,y_1\},...,\max \{x_m,y_m\}).$$

Then, $\m x\wedge\m  y$ is  the greatest lower bound and $\m x\vee\m  y$ is the least upper bound of 
the set $\{\m x,\m y\}$ with respect to the coordinate-wise order. Thus, $\leq$ is a lattice order in $\R^m.$
The operations $\wedge$ and $\vee$ are called \emph{lattice operations}.

A subset $M\subset \R^m$ is called a \emph{sublattice of
the coordinate-wise ordered Euclidean space} $\R^m$, if from
$\m x,\m y\in M$ it follows that $\m x\wedge\m  y,\;\m x\vee\m  y\in M.$ 

Denote by $P_D$ 
the projection mapping onto a nonempty closed convex set $D\subset \R^m,$ 
that is the mapping which associates
to $\m x\in \R^m$ the unique nearest point of $x$ in $D$ (\cite{Zarantonello1971}):

\[ P_D\m x\in D,\;\; \textrm{and}\;\; \|\m x-P_D\m x\|= \inf \{\|\m x-\m y\|: \;\m y\in D\}. \]

The nearest point $P_D\m x$ can be characterized by

\begin{equation}\label{charac}
P_D\m x\in D,\;\;\textrm{and}\;\;\lng P_D\m x -\m x,P_D\m x-\m y\rng \leq 0 ,\;\forall\m  y\in D.
\end{equation}

From the definition of the projection and the characterization (\ref{charac}) there follow immediately 
the relations: 

\begin{equation}\label{en}
	P_D(-\m x)=-P_{-D}\m x ,
\end{equation}

\begin{equation}\label{et}
	P_{\m x+D}\m y=\m x+P_D(\m y-\m x) 
\end{equation}
for any $\m x,\m y\in\R^m$,


For a closed convex cone $K$ we define the following operations in $\R^m$: 
\[\m x\sa_K\m  y=P_{\m x-K}\m y,\;\,\textrm{and}\;\;\m  x\su_K\m  y=P_{\m x+K}\m y\]
(see \cite{GowdaSznajderTao2004}). Assume the operations $\su_K$ and $\sa_K$ have precedence over the addition of vectors and 
multiplication of vectors by scalars.
 
A direct checking yields that if $K=\R^m_+$, then $\sa_K =\wedge$, and $\su_K =\vee$.
That is $\sa_K$ and $\su_K$ are some \emph{generalized lattice operations}.
Moreover: $\sa_K$ and $\su_K$ \emph{are lattice operations if and only if
the self-dual cone used in their definitions is a nonnegative orthant of
some Cartesian reference system.}
This suggest to call the operations $\sa_K$ and $\su_K$ \emph{lattice-like operations},
while a subset $M\subset \R^m$ which is \emph{invariant with respect to
$\sa_K$ and $\su_K$}  (i.e. if for any $\m x,\m y\in M$ we have $\m x\sa_K\m  y,\;\m x\su_K\m  y\in M$),
a \emph{lattice-like} or simply an \emph{l-l} subset of $(\R^m,K)$.

The following assertions are direct consequences of the definition
of lattice-like operations:

\begin{lemma}\label{l0}
	The following relations hold for any $\m x,\m y\in (\R^m,K)$: 
	$$\m x\sa_K\m  y=\m x-P_K(\m x-\m y),$$
	$$\m x\su_K\m  y=\m x+P_K(\m y-\m x).$$ 
\end{lemma}

If $K$ is a nonzero closed convex cone, then the closed convex set $C\subset \R^m$ is called a
\emph{$K$-isotone projection set} or simply \emph{$K$-isotone}
if $\m x\leq_K\m  y$ implies $P_C\m  x\leq_K\m  P_C\m  y$. In this case we use equivalently
the term \emph{$P_C$ is $K$-isotone}.

We shall refer next often to the following theorems:

\begin{theorem}\cite{NemethNemeth2012b}\label{ISOINV}
	Let $K\subset \R^m$ be a closed convex cone. Then, $C$ is a lattice-like set, if and only if $P_C$ is 
	$K$-isotone. 
\end{theorem}

\begin{theorem}\cite{NemethNemeth2012a}\label{FOOO}
The closed convex set $C$ with nonempty interior in $(\R^m,K)$ is lattice-like,
if and only if it is of form
\begin{equation*}
C=\bigcap_{i\in \n} H_-(\m u_i,\m a_i),
\end{equation*}
where each hyperplane $H(\m u_i,\m a_i)$ through $\m u_i$ with the normal $\m a_i$ is tangent to $C$ and is lattice-like.
\end{theorem}

\section{Characterization of the lattice-like subspaces of $(\R^m,K)$}

Denote by $K$ a closed convex cone in $\R^m$ and by $(\R^m,K)$ the
resulting ordered vector space.

The notation $G\Subset H$ will mean
\emph{$H$ and $G$ are subspaces of $\R^m$ and $G$ is  a subspace of $H$}.
Let $H\Subset \R^m$ and $L\subset H$ a closed convex cone.
The notation $G\sqsubset_L H$ will mean \emph{$G$ is an l-l subspace of $(H,L)$}.

We gather some results from Theorem 1 \cite{NemethNemeth2012b} and Lemma 6 \cite{NemethNemeth2012a} and particularize them for subspaces:

\begin{corollary}\label{fooalt}

Let $H$ a subspace in $(\R^m,K)$. 
the following assertions are equivalent:
\begin{enumerate}
	\item $H\sqsubset_K \R^m$,
	\item $P_{K} H\subset H$,
	\item $P_{H} K \subset K$.
\end{enumerate}

\end{corollary}

\begin{proof}
The corollary is in fact a reformulation of Theorem \ref{ISOINV} for
the case of $D=H$ a subspace. Indeed, condition 2 is nothing else as
the l-l property of of $H$ since if $\m x,\m y\in H$, then by
Lemma \ref{l0}, one has
$$\m x\sa_K\m y=\m x-P_K(\m x-\m y)\in H,$$
since $\m x,\,\m x-\m y,\,P_K(\m x-\m y)\in H.$

Similarly, $\m x\su_K\m y\in H.$

Condition 3 expresses, by the linearity  of $P_H$ its $K$-isotonicity.
\end{proof}

\begin{corollary}\label{fooaltkov}
	Let $G\Subset H$ and $H\sqsubset_K \R^m$. Then, $G\sqsubset_{K\cap H} H \Leftrightarrow G\sqsubset_K \R^m$.
\end{corollary}

\begin{proof}
In our proof we shall use without further comments the equivalences in Corollary \ref{fooalt}. 

Let $G\Subset H$ and $H\sqsubset_K \R^m$.

First suppose that $$G\sqsubset_K \R^m,$$ which is equivalent to  $$P_GK\subset K.$$ Hence, 
$$P_G(H\cap K)\subset P_GK\subset H\cap K,$$ since $P_G(K)\subset G\subset H.$ Thus, $G\sqsubset_{K\cap H}H$. 

Conversely, assume that $G\sqsubset_{H\cap K} H$. Take $\m x,\,\m y\in \R^m$ with $\m  x\leq_K\m  y.$ Then, from $H\sqsubset_K \R^m$ we have

$$P_H\m x\leq_K P_H\m y,\;\textrm{that is},\;P_H\m  y-P_H\m  x\in H\cap K,$$
hence 
$$P_H\m x\leq_{H\cap K}P_H\m y.$$
From $G\sqsubset_{H\cap K} H$ it follows that
$$ P_G P_H\m  x\leq_{H\cap K} P_G P_H\m  y.$$
From the property of orthogonal projections one has
$$P_G=P_G P_H.$$
Thus, the above relation writes as
$$P_G\m  x\leq_{H\cap K} P_G\m y,$$
or $P_G\m y -P_G\m x\in H\cap K \subset K.$
That is, 
$$P_G\m x\leq _K P_G\m y,$$
which shows that $G\sqsubset_K \R^m.$

\end{proof}

The following lemma is a direct consequence of
Lemma 3 and Lemma 8 in \cite{NemethNemeth2012b}:

\begin{lemma}\label{ti}
Suppose that $K$ is a closed convex cone in $\R^m$.
	Let $H(\m 0,\m a)\subset\R^m$ be a hyperplane 
	through the origin with unit normal vector $\m a\in\R^m$. Then, the following 
	assertions are equivalent:
	
	(i) $P_{H(\m 0,\m a)}$ is $K$-isotone;
	
	(ii) $P_{H(\m b,\m a)}$ is $K$-isotone for any $\m b\in \R^m$;
	
	(iii)  \[\lng\m  x,\m y\rng\ge\lng\m  a,\m x\rng\lng\m  a,\m y\rng,\] for any $\m x,\m y\in K$.
\end{lemma}

\section{Lattice-like subspaces of the Euclidean Jordan algebra}

In the particular case of a self-dual cone $K\subset \R^m$, J. Moreau's theorem (\cite{Moreau1962}) reduces to
the following lemma:

\begin{lemma}\label{lm}
Let $K\subset \R^m$ be a self-dual cone. Then, for any $\m x\in \R^m$
the following two conditions are equivalent:

(i) $\m x=\m u-\m v,\;\m u,\m v\in K, \lng\m  u,\m v\rng =0,$

(ii) $\m u=P_K\m x,\;\m v=P_K(-\m x).$
 
\end{lemma}

In all what follows we will consider that the ordered Euclidean space is
$(V,Q)$, the Euclidean Jordan
algebra $V$ of unit $\m e$ ordered by the cone $Q$ of squares in $V$. All the terms concerning $V$
will be equally used for $(V,Q)$.

Since the hyperplanes in Theorem \ref{FOOO} play an important role,
and since the l-l property is invariant with respect to translations
(Lemma 3, \cite{NemethNemeth2012a}), it is natural to study the l-l subspaces in $V$ which are naturally connected with  
the algebraic structure of this space.

\begin{theorem}\label{mell}

Any Jordan subalgebra of $(V,Q)$ is a lattice-like subspace.

\end{theorem}

\begin{proof}
 
Take a Jordan subalgebra $L$ in $V$ and denote by
$Q_0$ its cone of squares. We have
\begin{equation}\label{omeg}
Q_0=\{\m x^2:\,x\in L\}\subset \{\m x^2:\,x\in V\}=Q.
\end{equation}
We shall prove that 
\begin{equation}\label{omega=}
\m x\in L \Rightarrow P_{Q}\m x=P_{Q_0}\m x \in L.
\end{equation}
Indeed, we have, by Lemma \ref{lm} applied in the ordered
vector space $(L, Q_0)$, that
\begin{equation}\label{e8b}
	\m x=P_{Q_0}\m x-P_{Q_0}(-\m x),\;\; \lng P_{Q_0}\m x,P_{Q_0}(-\m x)\rng =0,
\end{equation}

By (\ref{omeg}) 
$$P_{Q_0}\m x,\;\; P_{Q_0}(-\m x)\in Q_0 \subset Q,$$
which, by equations \eqref{e8b} and Lemma \ref{lm}, yield $P_{Q_0}\m x=P_{Q}\m x$, or equivalently 
(\ref{omega=}).

Accordingly $P_QL\subset L,$ which by Corollary \ref{fooalt},
translates into $L\sqsubset_Q V.$
\end{proof}

\section{The Pierce decomposition of the Euclidean Jordan algebra and its lattice-like subspaces}
\label{pierce}

Let $r$ be the rank of $V$  and
$\{\m c_1,\dots,\m c_r\}$ be an arbitrary Jordan frame in $V$, that is, $\m c_k$ are primitive idempotents such that 
$$\m c_i\m c_j=0,\;\; \textrm{if}\;\; i\not=j,\;\; \m c_i^2=\m c_i,$$
$$\m c_1+\dots +\m c_r=\m e.$$
With the notation
$$V_{ii}=V(\m c_i,1)=\R \m c_i,$$
$$V_{ij}= V\lf(\m c_i,\frac{1}{2}\rg)\cap V\lf(\m c_j,\frac{1}{2}\rg),$$
(where for $\lambda \in \R$, $V(\m c_i, \lambda)=\{\m x\in V:\,\m c_i\m x=\lambda\m  x\}$),
we have by Theorem IV.2.1. \cite{FarautKoranyi1994} the following orthogonal
decomposition (the so-called \emph{Pierce decomposition}) of $V$:
\begin{equation}\label{Oplus}
V=\bigoplus_{i\leq j} V_{ij},
\end{equation}
where
\begin{equation}\label{Oplus1}
V_{ij}V_{ij}\subset V_{ii}+V_{jj};\;
V_{ij}V_{jk}\subset V_{ik},\;\;\textrm{if}\;\;i\not= k;\;
V_{ij}V_{kl}=\{0\}, \;\;\textrm{if}\;\; \{i,j\}\cap \{k,l\}=\emptyset.
\end{equation}

Taking for $1\leq k< r$
\begin{equation}\label{Oplus2}
V^{(k)}= \bigoplus_{i\leq j\leq k} V_{ij}
\end{equation}
is a Jordan algebra
with the unit
$$\m e_k=\m c_1+\dots +\m c_k.$$
Indeed, relations (\ref{Oplus1}) imply the invariance of $V^{(k)}$
with respect to the Jordan product. The same relations and the definitions imply
$\m e_k\m x_{ii}=\m c_i\m x_{ii}=\m x_{ii}$, for any $\m x_{ii}\in V_{ii}$ and $i\le k$; 
$\m c_lV_{ij}=\{0\}\textrm{ if }l\notin \{i,j\}$;
$\m e_k \m x_{ij}=(\m c_i+\m c_j)\m x_{ij}=\m x_{ij}$, for any $\m x_{ij}\in V_{ij}$ and $i,j\le k$, $i\ne j$. 
Hence $\m e_k$ is the unity of $V^{(k)}$. These relations also imply that
\begin{equation}\label{vkszubalg}
V^{(k)}=V(\m e_k,1)=\{\m x\in V:\,\m e_k\m x=\m x\}.
\end{equation}
Thus, $V(\m e_k,1)$ is a subalgebra (this follows also by Proposition IV.1.1 in \cite{FarautKoranyi1994} since
$\m e_k$ is idempotent). Hence
by Theorem \ref{mell}, $V(\m e_k,1)$ is an l-l subspace in $(V, Q)$.

A Jordan algebra is said \emph{simple} if it contains no nontrivial ideal.

A consequence of the above cited theorem and
the content of paragraph IV.2. of \cite{FarautKoranyi1994} is that
$V$ is simple if and only if $V_{ij}\not= \{0\}$
for any $V_{ij}$ in (\ref{Oplus}). By the same conclusion $V^{(k)}$ given by (\ref{Oplus2}) is simple too, and
by Corollary IV.2.6. in \cite{FarautKoranyi1994} the spaces $V_{ij},\;i\not= j$ have
the common dimension $d$, hence by (\ref{Oplus2}) 
$$\dme V^{(k)} =k+\frac{d}{2}k(k-1).$$

The subcone $F\subset Q$ is called a \emph{face of  $Q$} if whenever $0\leq_Q\m  x\leq_Q\m  y$ and $\m y\in F$
it follows that $\m x\in F$.

It is well known that for an arbitrary face $F$ of $Q$ one has $P_{\spa F}Q\subset Q$
(see e.g. Proposition II.1.3 in \cite{Iochum1984}). Hence by Corollary \ref{fooalt} it follows the assertion:

\begin{corollary}\label{facealg}
Each subspace generated by some face of $Q$ is a lattice-like subspace in $(V,Q)$.
\end{corollary}

We give an independent proof of this.

\begin{proof}
Let $\{\m c_1,...,\m c_r\}$ be a Jordan frame in $V$, $k\leq r$. If
$$\m e_k=\m c_1+\dots+\m c_k, \quad 0\leq k\leq r,$$
then by Theorem 3.1 in \cite{GowdaSznajder2006}
$$F=V(\m e_k,1)\cap Q=\{x\in Q:\,\m e_k\m  x=\m x\}$$
is a face of $Q$ and each face of $Q$ can be represented in this form for some Jordan frame.

The cone $F=V(\m e_k,1)\cap Q$ is the cone of squares in the subalgebra $V(\m e_k,1)$,
hence its relative interior is non-empty, accordingly

$$V(\m e_k,1)=\spa F=F-F.$$ 

Since $V(\m e_k,1)$ is a subalgebra, by Theorem \ref{mell} it is an l-l subspace.
\end{proof}

\section{The subalgebras and the lattice-like subspaces of
the space spanned by a Jordan frame}\label{orthant}

Suppose that the dimension of the Euclidean Jordan algebra $V$ is at least $2$. 
Let $\{\m c_1,...,\m c_r\}$ be a Jordan frame in $V$.
Then, 
$$V_{r}:=\spa\{\m c_1,...,\m c_r\}$$
is a Jordan subalgebra of $V$. Obviously, $V_r=V_{11}\oplus\dots\oplus V_{rr}$. If 
$\m x,\m y\in V_r$, then $$\m x\m y=(x_1y_1,...,x_ry_r),$$
where $x_i$ and $y_i$ are the coordinates of $\m x$, respectively $\m y$
with respect to the above Jordan frame.

By using the notations of the above section, denote $Q_r=Q\cap V_r$ and let us show that
$$Q_r=\cone \{\m c_1,...,\m c_r\}:=\lf\{\sum_{i=1}^r\lambda_i\m c_i:\lambda_i\ge0,\textrm{ }\forall 1\le i\le r\rg\}.$$

The inclusion $\cone\{\m c_1,...,\m c_r\}\subset Q_r$ is obvious. Next, we show that $Q_r\subset\cone\{\m c_1,...,\m c_r\}$. Suppose to the contrary, that
there exists $\m x\in Q_r\setminus\cone\{\m c_1,...,\m c_r\}$. It follows that $\lng \m c_k,\m x\rng <0$ for some $k\in\{1,\dots,r\}$. Since $Q$ is selfdual, 
this implies $\m x\notin Q$, which is a contradiction.

 The ordered vector space $(V_r,Q_r)$ can be considered
an $r$-dimensional Euclidean vector space ordered
with the positive orthant $Q_r$ engendered by the Jordan frame.

Let $H_{r-1}$ be an l-l hyperplane in $(V_r,Q_r)$, with the unit normal
$\m a\in V_r$. Thus, the results in \cite{NemethNemeth2012a} and \cite{NemethNemeth2012b} applies,
hence if 
\begin{equation}\label{a}
 \m a=(a_1,...,a_r),
\end{equation}
then we must have
\begin{equation}\label{aa}
a_ia_j\leq 0,\;\;\textrm{if}\;\; i\not=j.
\end{equation}

Then, there are two possibilities:
\begin{description}
	\item {\bf Case 1}. There exists an $i$ such that $a_i=1$ and $a_j=0$ for $j\not=i$.
	\item {\bf Case 2}. There are only two nonzero coordinates, say $a_k$ and $a_l$ with $a_ka_l<0$.
\end{description}
\begin{description}
	\item {\bf Ad 1}. In the {\bf Case 1} $$H_{r-1}=\spa\{\m c_1,...,\m c_{i-1},\m c_{i+1},...,\m c_r\}$$ and $H_{r-1}$ is obviously a Jordan algebra.
	\item {\bf Ad 2}. In the {\bf Case 2} $$H_{r-1} = \{\m x\in V_r: a_kx_k+a_lx_l=0\}.$$ We know from the above cited result, that $H_{r-1}$ is an l-l 
		subspace in $(V_r,Q_r)$ and since $V_r$ is a subalgebra of $V$, by Theorem \ref{mell},  $V_r\sqsubset_Q V$. By using Corollary 
		\ref{fooaltkov} we have, for the l-l subspace $H_{r-1}\sqsubset_{Q_r} V_r$, that 
		$$H_{r-1}\sqsubset_Q V.$$
\end{description}
In the case {\bf Ad 1} the l-l hyperplane $H_{r-1}$ is also a Jordan algebra. 

Suppose that {\bf Ad 2} holds. We would like to see under which condition
the l-l hyperplane $H_{r-1}$ is a Jordan algebra.

Let us suppose that $H_{r-1}$ is a Jordan algebra, and take
$\m x\in H_{r-1}$, $\m x=(x_1,...,x_r)$. Then, $\m x^2 =(x_1^2,...,x_r^2)\in H_{r-1}.$
Take $\m x$ with $x_l=a_k$ and $x_k=-a_l$. Then, $\m x\in H_{r-1}$ and we
must have $\m x^2\in H_{r-1}$. Hence
$$a_ka_l^2+a_la_k^2=a_ka_l(a_l+a_k)=0,$$
and since $a_ka_l\not=0,$ we must have
$$ a_k=\frac{\sqrt 2}{2}\qquad a_l=-\frac{\sqrt 2}{2},$$
or conversely.
In this case 
\begin{equation}\label{k=l}
 H_{r-1}= \{\m x:x_k=x_l\}
\end{equation}
is obviously a subalgebra.

\begin{remark}\label{nemalgebra} 
For any hyperplane $H_{r-1}$ in $V_r$ with the unit normal $\m a$ having only
two nonzero components with opposite signs and different absolute values $H_{r-1}$
is an l-l subspace, but not a Jordan subalgebra.
\end{remark}

If $\m a\in V$ is arbitrary, then there exists a Jordan frame $\{\m c_1,...,\m c_r\}$ such that
$\m a$ can be represented in the form (\ref{a}) (Theorem III.1.2 in \cite{FarautKoranyi1994}). We will call such a Jordan frame as
being \emph{attached to} $\m a$.

\begin{corollary}\label{aaa}
Let $H$ be a lattice-like hyperplane in $(V,Q)$ with the normal
$\m a$ and $\{\m c_1,...,\m c_r\}$ be a Jordan frame attached to it.
If $\m a$ is represented by  (\ref{a}), then the coordinates $a_i,\,i=1,...,r$
of $\m a$ satisfy the relations (\ref{aa}).
\end{corollary}
\begin{proof}
If $V_r=\spa\{\m c_1,...,\m c_r\}$, then $H_{r-1}=H\cap V_r$ is
an l-l hyperplane in $(V_r,Q_r)$ with the normal $\m a$ because it is the 
intersection of two l-l sets: $V_r$ (a subalgebra) and $H$. Thus, we
can apply the characterization of l-l hyperplanes in $(V_r,Q_r)$ described
above in this section.
\end{proof}

Denote by $\mc F (Q)$ the family of faces of $Q$, by $\mc A$ the family of subalgebras of $V$
and by $\mc L$ the family of the l-l subspaces in $V$. Then, by the above reasonings, we conclude the

\begin{corollary}
We have the following strict inclusions:
$$\{\spa F:\,F\in \mc F (Q)\} \subset \mc A \subset \mc L.$$
\end{corollary}
\begin{proof}
The second strict inclusion follows from Remark \ref{nemalgebra}.
The first inclusion is strict since for instance
the subspaces in (\ref{k=l}) are subalgebras which
are not generated by faces of $Q$. Indeed, take in $V_r$ the reference system
engendered by $\m c_1,...,\m c_r$ and let
$$H_{r-1}=\{(t,t,x_{r+3},...,x_r)\in V_r:t,x_j\in \R\}.$$
Take $\m y=(1,1,0,...,0)\textrm{ and }\m x=(1,0,0,...,0)\textrm{ in }Q_r=Q\cap V_r.$
Since in $V_r,\;\leq_Q=\leq_{Q_r}$ and the latter is coordinate-vise ordering,
$$0\leq_Q\m  x\leq_Q\m  y,$$
and we have $\m y\in H_{r-1}\cap Q$, but $\m x\notin H_{r-1}\cap Q$, which shows that
$H_{r-1}\cap Q$ is not a face.
\end{proof}

\section{ The inexistence
of lattice-like hyperplanes in\\ simple Euclidean Jordan algebras of rank $r\geq 3$.}

\begin{theorem}\label{nincs}
Suppose that $V$ is a simple Euclidean Jordan algebra of rank $r\geq 3$. 
Then, $V$ does not contain lattice-like hyperplanes.
\end{theorem}

\begin{proof}
Assume the contrary: $H$ is an l-l hyperplane through $0$ in $V$ with the
unit normal $\m a$.

Consider a Jordan frame $\{\m c_1,...,\m c_r\}$ attached to $\m a$.

The set
$$H_{r-1}=H\cap V_r$$
is obviously a hyperplane through $0$ in $V_r$.

Since by hypothesis $H\sqsubset_Q V$, by Corollary \ref{fooaltkov},  $H_{r-1}\sqsubset_{Q_r} V_r,$ 
where $Q_r=Q\cap V_r$.

If $\m a=(a_1,...,a_r)$ is the representation of $\m a$
in the reference system engendered by the Jordan frame,
then using Corollary \ref{aaa},
the l-l property of $H_{r-1}$ in $(V_r,Q_r)$ implies
that one of the following cases must hold:
\begin{description}
	\item {\bf Case 1}. For some $i$ $a_i=1$ and $a_j=0$ for $j\not=i$.
	\item {\bf Case 2}. There are only two nonzero coordinates, say $a_i$ and $a_j$ with $a_ia_j<0$.
\end{description}

Suppose that $i=1,\,j=2$.

Since $V$ is simple, $V_{12}\not= \{0\}$ (by Proposition IV.2.3 \cite{FarautKoranyi1994}), hence we
can take $\m x\in V_{12}$
with $\|\m x\|^2=2$. Then, by Exercise IV. 7 in \cite{FarautKoranyi1994}, 
we have that
\begin{equation}\label{xn}
 \m u=\frac{1}{2}\m c_1+\frac{1}{2}\m c_2+\frac{1}{2}\m x,\qquad \textrm{and} \qquad \m v=\frac{1}{2}\m c_1+\frac{1}{2}\m c_2-\frac{1}{2}\m x
\end{equation}
are idempotent elements, hence $\m u,\m v\in Q.$
We further have 
$$\m u\m v= \lf(\frac{1}{2}\m c_1+\frac{1}{2}\m c_2\rg)^2-\frac{1}{4}\m x^2,$$
whereby, by using Proposition IV.1.4 in \cite{FarautKoranyi1994}, we have \[\m x^2=\frac12\|\m x\|^2(\m c_1+\m c_2)=\m c_1+\m c_2,\]
and after raising to the second power and substitution
$$\m u\m v= \frac{1}{4}\m c_1+\frac{1}{4}\m c_2-\frac{1}{4}(\m c_1+\m c_2)=0.$$
Hence

$$\lng\m  u,\m v\rng=0.$$

Since $H_{r-1}$ is l-l, we have by Lemma \ref{ti}, 

$$0= \lng\m  u,\m v\rng \geq \lng \m a,\m u\rng \lng \m a,\m v\rng.$$

If $a_1=1, $ and $a_j=0$ for $j\not= 1$, the above relation becomes $0\geq \frac{1}{4}\|\m c_1\|^4,$ which is impossible.

Assume $a_1a_2<0$ and $a_j=0$ for $j>2$.

Take now
\begin{equation}\label{yn}
	\begin{array}{l}
		\m w=\ds\frac{1}{2}\m c_1+\frac{1}{2}\m c_3+\frac{1}{2}\m y\\\\
		\m z=\ds\frac{1}{2}\m c_1+\frac{1}{2}\m c_3-\frac{1}{2}\m y
	\end{array}
\end{equation}
with $\m y\in V_{13},\;\|\m y\|^2=2.$
Then, $\m w,\m z\in Q$ (similarly to $\m u,\m v\in Q$) and, by using the mutual orthogonality of the elements
$\m c_1,\,\m c_2,\,\m c_3,\,\m  y$ and Lemma \ref{ti}, it follows that
$$0=\lng\m  w,\m z\rng\geq\lng\m  a,\m w\rng\lng\m  a,\m z\rng =$$
$$\lng a_1\m c_1+a_2\m c_2,\frac{1}{2}\m c_1+\frac{1}{2}\m c_3+\frac{1}{2}\m y\rng\lng a_1\m c_1+a_2\m c_2,\frac{1}{2}\m c_1+\frac{1}{2}\m c_3-\frac{1}{2}\m y\rng
=\frac{1}{4}a_1^2\|\m c_1\|^2,$$
which is a contradiction.

\end{proof}

This theorem conferred with Theorem \ref{FOOO} and Lemma 3 in \cite{NemethNemeth2012a} yields the

\begin{corollary}
In the ordered Euclidean Jordan algebra $(V,Q)$ of rank at least $3$
there are no proper closed convex lattice-like set with nonempty interior.
In particular, for $n\geq 3$ the ordered space $(\sm,\smp)$
contains no proper, closed, convex lattice-like set with nonempty interior. 
\end{corollary}

\section{The case of the simple Euclidean Jordan algebras of rank 2}

A simple Euclidean Jordan algebra of rank 2 is isomorphic to an
algebra associated with a positive definite bilinear form
(Corollary IV.1.5 \cite{FarautKoranyi1994}). This
is in fact a Jordan algebra associated with the Lorentz cone. 
Hence the problem of the existence of l-l hyperplanes
in this case is answered positively in \cite{NemethNemeth2012a}
and \cite{NemethNemeth2012b}. In this section we use
the formalism developed in the preceding sections to this case too.
 \begin{lemma}\label{ketdim}
Suppose that $\m a$ is the unit normal to a lattice-like hyperplane $H$
through $0$ in the simple Euclidean Jordan algebra $V$ of
rank 2. Let $\{\m c_1,\m c_2\}$ be the Jordan frame attached to $\m a$ and $\m a=a_1\m c_1+a_2\m c_2$. Then, supposing $a_1>0$, we obtain
\begin{equation}\label{keta2}
 \m a= \frac{\sqrt 2}{2} \m c_1 -\frac{\sqrt 2}{2} \m c_2.
\end{equation}
\end{lemma}
\begin{proof}
Take $\m u$ and $\m v$ as in the formula (\ref{xn}). Then, $\m u,\m v\in Q$
and using Lemma \ref{ti} we obtain
\begin{gather*}
	0=\lng\m  u,\m v\rng\geq\lng\m  a,\m u\rng\lng\m  a,\m v\rng\\
	=\lng a_1\m c_1+a_2\m c_2,\frac{1}{2}\m c_1+\frac{1}{2}\m c_2+\frac{1}{2}\m x\rng 
	\lng a_1\m c_1+a_2\m c_2,\frac{1}{2}\m c_1+\frac{1}{2}\m c_2-\frac{1}{2}\m x\rng
	=\frac{1}{4}(a_1+a_2)^2,
\end{gather*}
whereby our assumption follows.
\end{proof}

\begin{theorem}\label{ketdim1}
Let $V$ be a simple Euclidean Jordan algebra of rank $2$ and $H$ be a hyperplane through $0$ with unit normal $\m a$ in $V$. Then, $H$
is lattice-like if and only if 
$\m a=\sqrt 2/2 \m c_1 -\sqrt 2/2 \m c_2$ in its Jordan frame representation.
In this case $H$ is a subalgebra. 
\end{theorem}
\begin{proof}
Suppose that $H=\ker\m  a,\;\|\m a\|=1$ is l-l,
and that the Jordan frame attached to $\m a$ is $\{\m c_1,\m c_2\}.$

Then, by Lemma \ref{ketdim}, it follows that $\m a$ is
of form (\ref{keta2}).

Suppose that the Jordan frame representation of $\m a$ is of form (\ref{keta2}). 
Then, equations \eqref{Oplus} and \eqref{Oplus1} imply that
$$ \ker\m  a=\{t(\m c_1+\m c_2)+\m x=t\m e+\m x:\,t\in \R, \,\m x\in V_{12}\}.$$
Then, for two arbitrary
elements $\m u,\m v\in \ker\m  a$, we have the representations:
$$\m u =t_1\m e+\m x;\textrm{ }\m v =t_2\m e+\m y;\textrm{ }\m x,\m y\in V_{12};\textrm{ }t_i\in \R,\textrm{ }i=1,2.$$
Then,
$$\m u\m v=t_1t_2\m e+t_1\m y+t_2\m x+ \m x\m y.$$
Since $\m x\m y=(1/4)((\m x+\m y)^2-(\m x-\m y)^2)$,
by using Proposition IV.1.4 in \cite{FarautKoranyi1994}, we conclude
that $\m x\m y=q(\m c_1+\m c_2)=q\m e$
with $q\in \R.$
Hence
$$\m u\m v=(t_1t_2+q)\m e+t_1\m y+t_2\m x\in \ker\m  a.$$
This shows that $H=\ker\m  a$ is a subalgebra, and hence an l-l set.
\end{proof}

\begin{remark}\label{ketketket}

With the notations in the above proof we have that $\spa \{\m c_1,\,\m c_2\}$ is a subalgebra of dimension $2$ in $V$.

Similarly to Remark \ref{nemalgebra}, it follows that there exist
l-l subspaces of dimension 1 in $\spa \{\m c_1,\,\m c_2\}$ which are not subalgebras.
\end{remark} 

Collating Theorem \ref{ketdim1} and Theorem \ref{FOOO} it follows the result:

\begin{corollary}\label{konvlor}
The closed convex set with nonempty
interior $M\subset V$  is a lattice-like set if and only if it is of the form:
\begin{equation*}
M=\bigcap_{i\in \n} H_-(\m u_i,\m a_i),
\end{equation*}
with the $\m a_i$ normal unit vectors represented in their Jordan
frame $\m c^i_1, \,\m c^i_2 $ by
\begin{equation}\label{keta}
	\m a_i=\varepsilon_i\left(\frac{\sqrt 2}{2}\m c^i_1-\frac{\sqrt 2}{2}\m c^i_2\right),\textrm{ }
	\varepsilon_i=1\textrm{ or }-1.
\end{equation}
\end{corollary} 

\begin{example}\label{lorentz}
Write the elements of $\R^{m+1}$ in the form
$(\m x,x_{m+1})$ with $\m x\in \R^m$ and $x_{m+1}\in \R.$
The Jordan product in $\R^{m+1}$ is defined by
$$(\m x,x_{m+1})\circ (\m y,y_{m+1})=(y_{m+1}\m x+x_{m+1}\m y, \lng\m  x,\m y\rng+x_{m+1}y_{m+1}),$$ 
where $\lng \m x,\m y\rng$ is the usual scalar product in $\R^m.$
The space $\R^{m+1}$ equipped with the usual scalar product and the
operation $\circ$ just defined becomes an Euclidean Jordan algebra of rank 2, denoted by 
$\mathcal L^{m+1}$ \cite{GowdaSznajderTao2004}, with the cone of squares $Q=\mathcal L^{m+1}_+$, the Lorentz cone
defined by (\ref{lorentzcone}).
 
The unit element in $\mathcal L^{m+1}$ is $(\m 0,1)$, where $\m 0$ is the zero
vector in $\R^m$.

The Jordan frame attached to $(\m x,x_{m+1})\in \mathcal L^{m+1}$ with $\m x\not=0$ is
$$\m c_1=\frac{1}{2}\left(\frac{\m x}{\|\m x\|},1\right),\quad \m c_2=
\frac{1}{2}\left(-\frac{\m x}{\|\m x\|},1\right).$$
The unit normal $\m a$ from Lemma \ref{ketdim} will be then parallel with
$(\m b,0)$ with some $\m b\in \R^m,\;\m b\not=\m 0.$
This means, that the hyperplanes $H(\m u_i,\m a_i)$ in Corollary \ref{konvlor}
are parallel with the $m+1$-th axis, and the closed convex set in
the corollary is in fact of form
$$M=C\times \R,$$
with $C$ closed convex set with nonempty interior in $\R^m.$

This is exactly the result in Example 1 of \cite{NemethNemeth2012b}.
\end{example}

\section{The general case}

For a general Euclidean Jordan algebra $V$, gathering the results
of Proposition III.4.4, Proposition III.4.5, and Theorem V.3.7,
of \cite{FarautKoranyi1994}, in Theorem 5 of \cite{GowdaSznajderTao2004} the following result is
stated:

\begin{theorem}\label{split}
Any Euclidean Jordan algebra $V$ is, in unique way, a direct sum 
\begin{equation}\label{dirossz}
V=\oplus_{i=1}^k V_i
\end{equation}
of simple Euclidean Jordan algebras $V_i,\,i=1,...,k$.
Moreover the cone of squares $Q$ in $V$ is, in a unique way,
a direct sum
\begin{equation}\label{kupdirossz}
Q=\oplus_{i=1}^k Q_k
\end{equation}
of the cones of squares $Q_i$ in $V_i,\,i=1,...,k$.
\end{theorem}

(Here the direct sum, (by a difference to that in the
Pierce decomposition), means Jordan-algebraic
and hence also orthogonal direct sum.)

Let $C\subset V$ a closed convex set.
From the results in Theorem \ref{split} it follows easily  (using the notations in the
theorem), that
\begin{equation}\label{projossz}
P_C = \sum_{i=1}^k P_{C_i},
\end{equation}
with $C_i=C\cap V_i,\;i=1,...,k$.

Collating these results with Corollary \ref{fooaltkov}, we have the following: 

\begin{corollary}\label{ossz}
With the notations in Theorem \ref{split},
for the subspace $M\Subset V$ we have the
equivalence:
\begin{equation}\label{llossz}
M\sqsubset_Q V \Leftrightarrow M\cap V_i\sqsubset_{Q_i} V_i,\;i=1,...,k.
\end{equation}
For the closed convex set $C$ the projection $P_C$
is $Q$-isotone if and only if $P_{C\cap V_i}$ is $Q_i$-isotone in $(V_i,Q_i),\;i=1,...k.$
\end{corollary}

\begin{corollary}\label{osszhyp}
If $H$ is a lattice-like hyperplane in $V$ represented as (\ref{dirossz}) in Theorem \ref{split},
then $V_i\Subset H$ for each simple subalgebra in (\ref{dirossz}) of rank at least $3$.
\end{corollary}

\begin{proof}
Assume the contrary. Then, $H\cap V_i$ is an l-l hyperplane in $V_i$, which
contradicts Theorem \ref{nincs}.
\end{proof}

Gathering the results in Theorem \ref{FOOO}, Section 7, Corollary \ref{konvlor} and Corollary \ref{osszhyp}
we have 

\begin{theorem}
Suppose that $V$ is an Euclidean Jordan algebra of form
(\ref{dirossz}) with $V_i$ simple subalgebras. Let us write this sum as
\begin{equation}\label{partition}
V=W_1\oplus W_2\oplus W_3
\end{equation}
where
\begin{equation}\label{partition1}
W_1=\oplus_{i\in I_1}V_i ,\;\;W_2=\oplus_{i\in I_2} V_i,\;\; W_3=\oplus_{i\in I_3} V_i,
\end{equation}
such that $V_i$ for $i\in I_1$ are the subalgebras of rank $1$, for $i\in I_2$,
the subalgebras of rank $2$, and for $i\in I_3$ the subalgebras of rank at least $3$.
Then, $C\subset V$ is a closed convex lattice-like subset with nonempty interior
if and only if the following conditions hold:
\begin{equation}\label{alap}
C=\bigcap_{i\in \n} H_-(\m u_i,\m a_i),
\end{equation}
where each hyperplane $H(\m u_i,\m a_i)$ through $\m u_i$ and with the unit normal $\m a_i$ is tangent to $C$ and is lattice-like.
Let $\{\m c^i_1,\dots,\m c^i_r\}$ be a Jordan frame attached to $\m a_i$.
The last conditions hold if and only if 
\begin{equation}\label{aire}
 \m a_i= a^i_1\m c^i_1+...+a^i_{r_1}\m c^i_{r_1} +a^i_{r_1+1}\m c^i_{r_1+1}+...+a^i_{r_2}\m c^i_{r_2}+a^i_{r_2+1}\m c^i_{r_2+1}+...+a^i_{r}\m c^i_{r}
\end{equation}
with $\m c^i_1,...,\m c^i_{r_1}\in W_1;\;\;\m c^i_{r_1+1},...,\m c^i_{r_2}\in W_2$, and 
$\m c^i_{r_2+1},...,\m c^i_r\in W_3$,
and exactly one of the following two cases hold:
\begin{enumerate}
	\item[(i)] There exists a $k\in\{1,\dots,r_1\}$ with $a^i_k\not= 0$ and exactly one of the 
		following two statements is true
		\begin{enumerate}
			\item[(i)'] The equality $a^i_j=0$ holds for $j\ne k$,
			\item[(i)''] There exists an $l\in\{1,\dots,r_1\}$ such that 
				$a^i_la^i_k<0$ and $a^i_j=0$, for $j\notin\{k,l\}$,
		\end{enumerate}
	\item[(ii)] There exists $k,l\in\{r_1+1,\dots,r_2\}$ and $p\in I_2$ such that $\m c^i_k,\m c^i_l\in V_p$, $a^i_k=\sqrt{2}/2$, 
		$a^i_l=-\sqrt{2}/2$ and $a^i_j=0$, for $j\notin\{k,l\}$.  
\end{enumerate}
\end{theorem}

\begin{proof}
Observe first that using the representation (\ref{aire}) of $a_i$ and
the partition (\ref{partition1}) of $V$, we have the following relations:
$$r_1=|I_1|,\;r_2-r_1=2|I_2|,\;r-r_2\geq 3|I_3|.$$

The representation (\ref{alap}) follows from Theorem \ref{FOOO}.
Let us see first that the alternative (i), respectively  (ii) is
 sufficient for $H(\m u_i,\m a_i)$ to be an l-l set.

If (i) holds, then the hyperplane $H_{r_1-1}$ through $0$ with the normal
$\m a^i_{W_1}=a^i_1\m c^i_1+...+a^i_{r_1}\m c^i_{r_1}$ is by Section 7 an l-l set in $\spa \{\m c^i_1,...,\m c^i_{r_1}\}$ ordered
by the orthant engendered by $\m c_1,...,\m c_{r_1}$.
Hence $H(\m u_i,\m a_i)=(H_{r_1-1}+\m u_i)\bigoplus W_2\bigoplus W_3$ is l-l in $V$ (by Theorem \ref{ISOINV} 
and Lemma \ref{ti}).

If (ii) holds, then the hyperplane $H'$ through $0$ with the normal $\m a^i_0=a^i_k\m c^i_k+a^i_l\m c^i_l=
(\sqrt 2/2)(\m c^i_k-\m c^i_l)$ in $V_p$ is l-l (by Theorem \ref{ketdim1}, Theorem \ref{ISOINV} and Lemma 
\ref{ti}), hence
$H(\m u_i,\m a_i)=(H'+\m u_i)\bigoplus (\bigoplus_{j\not= p}V_j)$ is l-l in $V$.

To complete the proof we have to show the necessity of the alternatives (i) and (ii).
Observe first that if $H(\m u_i,\m a_i)$ is l-l, then in the representation (\ref{aire}) 
of $\m a_i$, by Corollary \ref{osszhyp}, we must have $a^i_j=0$
whenever $j>r_2$. Thus, if $a^i_j\not =0$, then $j\leq r_2.$

Suppose that $a^i_k\not= 0$ for for some $\m c^i_k\in W_2$.
Then, there exists an $a^i_l\not= 0$ and $\m c^i_k,\,\m c^i_l\in V_p$, for some $V_p$ in the representation of $W_2$.
Indeed, in the case $a^i_k\not= 0$ it follows that $a^i_k\m c^i_k\in V_p\setminus \{0\}$
for some $V_p\subset W_2$, hence $H(\m u_i,\m a_i)\cap V_p$ 
is a hyperplane in $V_p$ and our assertion follows from Lemma \ref{ketdim} (and in particular one of
$a^i_k$ and $a^i_l$ is $\sqrt 2/2$ the other is $-\sqrt 2/2$).
From Corollary \ref{aaa} it follows then that
$a^i_j=0$ for $j\notin \{k,l\}$. Thus, the alternative (ii) must hold.

Suppose now that $a^i_j\not= 0$ for some $j\leq r_1$.
Then from the reasoning of the above paragraph and Corollary \ref{aaa}, we must have $a^i_k=0$ if $k>r_1$. 
In this  case two situations are possible: (i)' $a^i_j=1$ and $a^i_l=0$ for $l\not=j$,
and  (i)'' there exists $a^i_l\not= 0,\,(l\leq r_1)$ with $a^i_ja^i_l<0$ and $a^i_k=0$ for $k\notin \{j,l\}.$
Thus, the alternative (i) must hold.

\end{proof}


\begin{example}\label{simplesplit}
Let $V$ be a simple Euclidean Jordan algebra with
the Pierce decomposition given by (\ref{Oplus}) and (\ref{Oplus1}), and $d$ the common dimension of $V_{ij}$, $i\ne j$ (see Corollary IV.2.6 \cite{FarautKoranyi1994}).
Denote
$$W_{k,l}=\bigoplus_{k\leq i\leq j\leq l}V_{ij}.$$
Then, $W_{k,l}$ is a subalgebra, hence an l-l subspace.
The sum
$$W_{1,k}\bigoplus W_{k+1,r},\quad k<r$$
is a subalgebra too, and hence an l-l subspace.
Suppose that $r\geq 4$ and $2\leq k \leq r-2$. Let $H_0$
be an l-l hyperplane in $W_{k+1,r}$ which is not its subalgebra.
Then, 
$$W_{1,k}+H_0$$
is an l-l subspace in $V$ of dimension $k+(d/2)k(k-1)+ r-k-1$
which is not an algebra.

\emph{Question}: Is every l-l subspace of $V$ which is not a subalgebra of this type?
\end{example}

\section*{Acknowledgement} The authors express their gratitude to Roman Sznajder for his
helpful comments and information on Euclidean Jordan algebras.

\bibliographystyle{hieeetr}
\bibliography{semidef}

\end{document}